\documentclass[a4paper, 12pt]{amsart}
\usepackage[cp1251]{inputenc}
\usepackage[english]{babel}
\usepackage{amsmath,amssymb,a4wide,latexsym,amscd,amsthm}
\usepackage{graphics}

\theoremstyle{plain}
\newtheorem{theorem}{Theorem}[section]
\newtheorem{lemma}{Lemma}[section]

\theoremstyle{definition}
\newtheorem{example}{Example}[section]
\newtheorem{definition}{Definition}[section]
\newtheorem{remark}{Remark}

\numberwithin{equation}{section}

\begin{document}

\title[]{On the divergence of series of the form $\sum_{k=1}^\infty\|A_k x\|^p$}

\author{Ivan S. Feshchenko}

\maketitle

\begin{abstract}
Let $X$, $Y_k$, $k\geqslant 1$ be normed linear spaces, and $A_{k}:X\to Y_k$, $k\geqslant 1$, be continuous linear operators.
For $p\in[1,\infty]$, define the set
\begin{equation*}
\mathcal{D}_p=\{x\in X\mid (\|A_{1}x\|,\|A_{2}x\|,\ldots)\notin\ell_p\}.
\end{equation*}
We provide sufficient conditions for $\mathcal{D}_p$ to be dense in $X$, where $p\in[1,\infty]$ is fixed,
and for $\bigcap_{p\in[1,p_0)}\mathcal{D}_p$ to be dense in $X$, where $p_0\in(1,\infty]$ is fixed.

We also show that these results can not be improved (in a certain sense).


\end{abstract}

\section{Formulation of the problem}\label{S:formulation}

For $p\in[1,\infty]$, define $\ell_p$ to be the linear space of all sequences $a=(a_1,a_2,\ldots)$, $a_k\in\mathbb{R}$,
such that
$\sum_{k=1}^\infty|a_k|^p<\infty$ (if $p=\infty$, then $\sup_{k\geqslant 1}|a_k|<\infty$),
endowed with the norm
\begin{equation*}\
\|a\|_p=\left(\sum_{k=1}^\infty|a_k|^p\right)^{1/p}
\end{equation*}
(if $p=\infty$, then $\|a\|_\infty=\sup_{k\geqslant 1}|a_k|$).

Let $X$, $Y_k$, $k\geqslant 1$ be normed linear spaces over a field $\mathbb{K}$ of real or complex numbers,
and $A_{k}:X\to Y_k$, $k\geqslant 1$, be continuous linear operators.
For $p\in[1,\infty]$, define the set
\begin{equation*}
\mathcal{D}_p=\{x\in X\mid (\|A_{1}x\|,\|A_{2}x\|,\ldots)\notin\ell_p\}.
\end{equation*}

Clearly, $\mathcal{D}_{p_2}\subset\mathcal{D}_{p_1}$ if $p_2>p_1$.
Moreover, if $\mathcal{D}_p$ is nonempty, then $\mathcal{D}_p$ is dense in $X$ (see Lemma~\ref{L:lemma_dense_D}).

We provide sufficient conditions for $\mathcal{D}_p$ to be dense in $X$, where $p\in[1,\infty]$ is fixed (see Theorem~\ref{T:theorem1}),
and for $\bigcap_{p\in[1,p_0)}\mathcal{D}_p$ to be dense in $X$, where $p_0\in(1,\infty]$ is fixed (see Theorem~\ref{T:theorem2}).

In Section~\ref{S:sharpness} we show that these results can not be improved (in a certain sense).

To formulate our results we need some auxiliary notions.

\section{$M$-cotype of a normed linear space}\label{S:M-cotype}

Let us recall the definition a normed linear space of $M$-cotype $\rho$~\cite[Definition 4.2.2]{Kadets}
(in this book the definition is given only for real spaces).
Note that the notion of $M$-cotype arises naturally in the study of various geometric properties of Banach spaces
(see, e.g.,~\cite[Sections 4.2, 5.2]{Kadets}).

Let $V$ be a normed linear space over a field $\mathbb{K}$ of real or complex numbers, and $\rho\in[1,\infty)$.

\begin{definition}\label{D:M-cotype}
The space $V$ is said to have \emph{$M$-cotype $\rho$ with constant $C>0$} if
\begin{equation*}
\max\left\{\|\sum_{k=1}^n \varepsilon_k v_k\| \mid \varepsilon_k=\pm 1\right\}\geqslant
C\left(\sum_{k=1}^n\|v_k\|^\rho\right)^{1/\rho}
\end{equation*}
for any $n\in\mathbb{N}$ and $v_1,\ldots,v_n\in V$.
The space $V$ is said to have \emph{$M$-cotype $\rho$} if there exists a constant $C>0$ such that
$V$ has $M$-cotype $\rho$ with constant $C$.
\end{definition}

\begin{remark}\label{R:complex_case}
For the case $\mathbb{K}=\mathbb{C}$ it is natural to give the following definition.
The space $V$ is said to have $M$-cotype $\rho$ if there exists a constant $C>0$ such that
\begin{equation*}
\max\left\{\|\sum_{k=1}^n \alpha_k v_k\| \mid \alpha_k\in\mathbb{C}, |\alpha_k|=1\right\}\geqslant
C\left(\sum_{k=1}^n\|v_k\|^\rho\right)^{1/\rho}
\end{equation*}
for any $n\geqslant 1$ and $v_1,\ldots,v_n\in V$.
This definition is equivalent to the definition above.
This follows from the inequality
\begin{equation}\label{E:ineq_complex_case}
\max\left\{\|\sum_{k=1}^n \alpha_k v_k\| \mid \alpha_k\in\mathbb{C}, |\alpha_k|=1\right\}\leqslant
2
\max\left\{\|\sum_{k=1}^n \varepsilon_k v_k\| \mid \varepsilon_k=\pm 1\right\}
\end{equation}
which is valid for any $v_1,\ldots,v_n\in V$.
Let us prove this inequality.
First, note that the function $\|\sum_{k=1}^n t_k v_k\|$ is convex in $(t_1,\ldots,t_n)\in\mathbb{R}^n$.
It follows that
\begin{equation*}
\|\sum_{k=1}^n t_k v_k\|\leqslant\max\left\{\|\sum_{k=1}^n \varepsilon_k v_k\| \mid \varepsilon_k=\pm 1\right\}
\end{equation*}
for any $t_k\in\mathbb{R}$, $|t_k|\leqslant 1$.
For any $\alpha_k\in\mathbb{C}$, $|\alpha_k|=1$, $1\leqslant k\leqslant n$, we have
\begin{align*}
&\|\sum_{k=1}^n\alpha_k v_k\|=\|\sum_{k=1}^n\mathrm{Re}(\alpha_k)v_k+i\sum_{k=1}^n\mathrm{Im}(\alpha_k)v_k\|\leqslant\\
&\leqslant\|\sum_{k=1}^n\mathrm{Re}(\alpha_k)v_k\|+\|\sum_{k=1}^n\mathrm{Im}(\alpha_k)v_k\|\leqslant
2 \max\left\{\|\sum_{k=1}^n \varepsilon_k v_k\| \mid \varepsilon_k=\pm 1\right\}.
\end{align*}
This proves~\eqref{E:ineq_complex_case}.
\end{remark}

\begin{remark}
A Banach space $V$ is said to have \emph{cotype} $\rho$ (see, e.g. \cite[Section 5.3]{Kadets}) if there exists a constant $C$ such that
\begin{equation*}
\mathbb{E}\|\sum_{k=1}^n r_k v_k\|\geqslant C\left(\sum_{k=1}^n\|v_k\|^\rho\right)^{1/\rho}
\end{equation*}
for any $n\geqslant 1$ and $v_1,\ldots,v_n\in V$, where $r_1,r_2,\ldots$ is a sequence of independent random variables that
take the values $\pm 1$ with the equal probabilities $\mathbb{P}(r_k=1)=\mathbb{P}(r_k=-1)=1/2$, and where $\mathbb{E}$ denotes the expectation.
Clearly, if $V$ has cotype $\rho$, then $V$ has $M$-cotype $\rho$.
\end{remark}

Let us provide some examples.

\begin{example}
Let $V$ be a finite dimensional space.
It is easy to check that $V$ has $M$-cotype $\rho=1$.
\end{example}

\begin{example}
Let $V$ be a Hilbert space.
It is easy to check that $V$ has $M$-cotype $\rho=2$ with constant $C=1$.
\end{example}

\begin{example}\label{EX:M_cotype_measure_space}
Suppose $(T,\mathcal{F},\mu)$ is a measure space, and $s\in[1,\infty)$.
Let $V=L_s(T,\mathcal{F},\mu)$.
Then $V$ has $M$-cotype $\rho=\max\{2,s\}$
(see, e.g.,~\cite[Proof of Theorem 4.2.1]{Kadets}. The proof is given for the case $\mathbb{K}=\mathbb{R}$, but
it is also valid for $\mathbb{K}=\mathbb{C}$).
\end{example}

\section{Main results}\label{S:main_results}

First, we give a sufficient condition for $\mathcal{D}_p$ to be dense in $X$.

To formulate this result, we need a few auxiliary definitions.
For a normed linear space $V$ over a field $\mathbb{K}$,
define $V^*$ to be the linear space of all continuous linear mappings $v^*:V\to\mathbb{K}$,
endowed with the norm
\begin{equation*}
\|v^*\|=\sup_{v\in V, \|v\|=1}|v^*(v)|.
\end{equation*}
For two normed linear spaces $V,W$ define $\mathcal{B}(V,W)$ to be the linear space of all continuous linear operators $A:V\to W$,
endowed with the norm
\begin{equation*}
\|A\|=\sup_{v\in V, \|v\|=1}\|Av\|.
\end{equation*}
In what follows we set $1/0=\infty$ and $1/\infty=0$.

\begin{theorem}\label{T:theorem1}
Let $X$ be a Banach space, and $Y_k$, $k\geqslant 1$ be normed linear spaces.
Let $A_k\in\mathcal{B}(X,Y_k)$, $k\geqslant 1$.

Suppose that $X^*$ has $M$-cotype $\rho\in[1,\infty)$.
Let $p\in[1,\rho/(\rho-1)]$.
Define $r\in[\rho,\infty]$ by
\begin{equation}\label{E:equation_for_r}
\frac{1}{p}-\frac{1}{r}=1-\frac{1}{\rho}.
\end{equation}

If
\begin{equation*}
(\|A_{1}\|,\|A_{2}\|,\ldots)\notin \ell_r,
\end{equation*}
then $\mathcal{D}_p$ is dense in $X$.
\end{theorem}

Now we give a sufficient condition for $\bigcap_{p\in[1,p_0)}\mathcal{D}_p$ to be dense in $X$.

\begin{theorem}\label{T:theorem2}
Let $X$ be a Banach space, and $Y_k$, $k\geqslant 1$ be normed linear spaces.
Let $A_k\in\mathcal{B}(X,Y_k)$, $k\geqslant 1$.

Suppose that $X^*$ has $M$-cotype $\rho\in[1,\infty)$.
Let $p_0\in(1,\rho/(\rho-1)]$.
Define $r_0\in(\rho,\infty]$ by
\begin{equation*}
\frac{1}{p_0}-\frac{1}{r_0}=1-\frac{1}{\rho}.
\end{equation*}

If
\begin{equation*}
(\|A_{1}\|,\|A_{2}\|,\ldots)\notin \ell_{r}\quad\text{for}\quad r\in[\rho,r_0),
\end{equation*}
then $\bigcap_{p\in[1,p_0)}\mathcal{D}_p$ is dense in $X$.
\end{theorem}

\section{Sharpness of Theorem~\ref{T:theorem1}}\label{S:sharpness}

In this section we show that Theorem~\ref{T:theorem1} is sharp, that is, the condition
\begin{equation*}
(\|A_{1}\|,\|A_{2}\|,\ldots)\notin \ell_{r}
\end{equation*}
is necessary for $\mathcal{D}_p$ to be dense in $X$.
More precisely, in examples below
for any sequence of nonnegative numbers
$a_k$, $k\geqslant 1$, such that $(a_{1},a_{2},\ldots)\in \ell_r$
we construct operators $A_{k}\in\mathcal{B}(X,Y_k)$ such that
$\|A_k\|=a_k$, $k\geqslant 1$, and $\mathcal{D}_p=\varnothing$.

To be specific, we assume that $\mathbb{K}=\mathbb{R}$.

\begin{example}
Let $X=\mathbb{R}$.
Then $X^*=\mathbb{R}$.
Hence, $X^*$ has $M$-cotype $\rho=1$.
We have $\rho/(\rho-1)=\infty$.
Let $p\in[1,\infty]$.
By~\eqref{E:equation_for_r} we get $r=p$.
Suppose that $a_k\geqslant 0$, $k\geqslant 1$, and $(a_{1},a_{2},\ldots)\in\ell_r$.
Define $A_k:\mathbb{R}\to\mathbb{R}$, $k\geqslant 1$, by
\begin{equation*}
A_k x=a_k x,\qquad x\in\mathbb{R}.
\end{equation*}
Clearly, $\|A_k\|=a_k$, $k\geqslant 1$, and $\mathcal{D}_p=\varnothing$.
\end{example}

\begin{example}
Let $X=\ell_s$, where $s\in(1,2]$.
Then $X^*=\ell_t$, where $t\in[2,\infty)$ is defined by $1/s+1/t=1$.
Hence, $X^*$ has $M$-cotype $\rho=t$ (see Example~\ref{EX:M_cotype_measure_space}).
We have $\rho/(\rho-1)=s$.
Let $p\in[1,s]$.
Then $r$ is defined by $1/p-1/r=1/s$.
Suppose that $a_k\geqslant 0$, $k\geqslant 1$, and $(a_{1},a_{2},\ldots)\in\ell_r$.
Define $A_k:\ell_s\to\mathbb{R}$, $k\geqslant 1$, by
\begin{equation*}
A_k x=a_k x_k,\qquad x=(x_1,x_2,\ldots)\in\ell_s.
\end{equation*}
Clearly, $\|A_k\|=a_k$, $k\geqslant 1$.
Let us show that $\mathcal{D}_p=\varnothing$.
Consider any $x=(x_{1},x_{2},\ldots)\in\ell_s$.
We have
\begin{equation*}
\sum_{k=1}^\infty |A_k x|^p=\sum_{k=1}^\infty(a_{k}|x_k|)^p.
\end{equation*}
Since $(a_1,a_2,\ldots)\in\ell_r$, $(|x_1|,|x_2|,\ldots)\in\ell_s$, and $1/r+1/s=1/p$,
we conclude that $(a_1|x_1|,a_2|x_2|,\ldots)\in\ell_p$.
Hence, $\mathcal{D}_p=\varnothing$.
\end{example}

\begin{example}
For $s\in[1,\infty)$, define $L_s=L_s([0,1],dx)$.
Let $X=L_s$, where $s\in[2,\infty)$.
Then $X^*=L_t$, where $t\in(1,2]$ is defined by $1/s+1/t=1$.
Hence, $X^*$ has $M$-cotype $\rho=2$ (see Example~\ref{EX:M_cotype_measure_space}).
We have $\rho/(\rho-1)=2$.
Let $p\in[1,2]$.
Then $r$ is defined by $1/p-1/r=1/2$.
Suppose that $a_k\geqslant 0$, $k\geqslant 1$, and $(a_{1},a_{2},\ldots)\in\ell_r$.

Let $r_k(t)$, $k\geqslant 1$, be the Rademacher functions,
\begin{equation*}
r_k(t)=sign\sin 2^{k}\pi t,\qquad t\in[0,1],\qquad k\geqslant 1.
\end{equation*}
It is well-known that the system $r_k$, $k\geqslant 1$, is an orthonormal system in $L_2$, that is,
$\langle r_k,r_m\rangle_{L_2}=0$ for $k\neq m$ and $\|r_k\|_{L_2}=1$ for $k\geqslant 1$, where
$\langle f(t),g(t)\rangle_{L_2}=\int_{[0,1]}f(t)g(t)\,dt$, $f,g\in L_2$.
Define $A_k\in\mathcal{B}(L_s,\mathbb{R})$, $k\geqslant 1$, by
\begin{equation*}
A_k x=a_k\langle x,r_k\rangle_{L_2}=a_k\int_{[0,1]}x(t)r_k(t)\,dt,\qquad x\in L_s.
\end{equation*}
We claim that $\|A_k\|=a_k$, $k\geqslant 1$.
Indeed, we have $|A_k x|\leqslant a_k\|x\|_{L_1}\leqslant a_k\|x\|_{L_s}$, hence, $\|A_k\|\leqslant a_k$.
Moreover, $A_k r_k=a_k$ and $\|r_k\|_{L_s}=1$, hence $\|A_k\|\geqslant a_k$.
Consequently, $\|A_k\|=a_k$.

Let us show that $\mathcal{D}_p=\varnothing$.
Consider any $x\in L_s$.
We have
\begin{equation*}
\sum_{k=1}^\infty|A_k x|^p=\sum_{k=1}^\infty\left(a_k\left|\langle x,r_k \rangle_{L_2}\right|\right)^p.
\end{equation*}
Since $L_s\subset L_2$ and $r_k$, $k\geqslant 1$, is an orthonormal system in $L_2$, we conclude that
\begin{equation*}
\left(\left|\langle x,r_1\rangle_{L_2}\right|,\left|\langle x,r_2\rangle_{L_2}\right|,\ldots\right)\in\ell_2.
\end{equation*}
Since $(a_1,a_2,\ldots)\in\ell_r$ and $1/2+1/r=1/p$, we conclude that
\begin{equation*}
(a_1\left|\langle x,r_1 \rangle_{L_2}\right|,a_2\left|\langle x,r_2 \rangle_{L_2}\right|,\ldots)\in\ell_p.
\end{equation*}
Hence, $\mathcal{D}_p=\varnothing$.
\end{example}

\section{Auxiliary notions and results}\label{S:aux_results_notions}

\subsection{Auxiliary notions}\label{SS:aux_notions}

Let $n\in\mathbb{N}$ and $s\in[1,\infty]$.
For a vector $a=(a_1,\ldots,a_n)$, $a_k\in\mathbb{R}$, define
\begin{equation*}
\|a\|_s=
\begin{cases}
\left(\sum_{k=1}^n |a_k|^s\right)^{1/s},& \text{if}\quad s\in[1,\infty),\\
\max_{1\leqslant k\leqslant n}|a_k|,&\text{if}\quad s=\infty.
\end{cases}
\end{equation*}

Let $X_k$, $1\leqslant k\leqslant n$, be normed linear spaces.
Define $\ell_s(X_1,\ldots,X_n)$ to be the linear space of all $n$-tuples
\begin{equation*}
x=(x_1,\ldots,x_n),\quad x_k\in X_k,\quad 1\leqslant k\leqslant n,
\end{equation*}
endowed with the norm
\begin{equation*}
\|x\|_s=\|(\|x_1\|,\ldots,\|x_n\|)\|_s.
\end{equation*}
It is easy to check that
\begin{equation*}
(\ell_s(X_1,\ldots,X_n))^*=\ell_t(X_1^*,\ldots,X_n^*),
\end{equation*}
where $t\in[1,\infty]$ is defined by $1/s+1/t=1$.
Note that
\begin{equation*}
x^*(x)=\sum_{k=1}^n x_k^*(x_k)
\end{equation*}
for $x^*=(x_1^*,\ldots,x_n^*)\in\ell_t(X_1^*,\ldots,X_n^*)$ and $x=(x_1,\ldots,x_n)\in\ell_s(X_1,\ldots,X_n)$.

\subsection{Auxiliary results}\label{SS:aux_results}

\begin{lemma}\label{L:lemma_dense_D}
Let $X$ be a normed linear space.
Suppose $\mathcal{D}$ is a nonempty subset of $X$ such that $X\setminus\mathcal{D}$ is a linear set.
Then $\mathcal{D}$ is dense in $X$.
\end{lemma}
\begin{proof}
Fix $d\in\mathcal{D}$.
Consider any $x\in X$.
It is easily seen that $|\{\lambda\mid x+\lambda d\notin\mathcal{D}\}|\leqslant 1$.
Since $x+\lambda d\to x$ as $\lambda\to 0$, we conclude that $\mathcal{D}$ is dense in $X$.
\end{proof}

\begin{lemma}\label{L:unbounded}
Let $X$ be a Banach space, and $Y_\gamma$, $\gamma\in\Gamma$ be normed linear spaces.
Let $A_\gamma\in\mathcal{B}(X,Y_\gamma)$, $\gamma\in\Gamma$.
If
\begin{equation*}
\sup_{\gamma\in\Gamma}\|A_\gamma\|=\infty,
\end{equation*}
then there exists $x\in X$ such that
\begin{equation*}
\sup_{\gamma\in\Gamma}\|A_\gamma x\|=\infty.
\end{equation*}
\end{lemma}

This lemma is a direct consequence of the following lemma which is a generalization of the principle of uniform boundedness.

\begin{lemma}\label{L:uniform_boundedness}
Let $X$ be a Banach space, and $Y_\gamma$, $\gamma\in\Gamma$ be normed linear spaces.
Let $A_\gamma\in\mathcal{B}(X,Y_\gamma)$, $\gamma\in\Gamma$.
If
\begin{equation*}
\sup_{\gamma\in\Gamma}\|A_\gamma x\|<\infty
\end{equation*}
for any $x\in X$, then
\begin{equation*}
\sup_{\gamma\in\Gamma}\|A_\gamma\|<\infty.
\end{equation*}
\end{lemma}

The proof is exactly the same as of the principle of uniform boundedness~\cite[Theorem 3.11]{MacCluer}.

The following lemma plays a crucial role in the proof of Theorems~\ref{T:theorem1}, \ref{T:theorem2}.

\begin{lemma}\label{L:main_lemma}
Let $n\in\mathbb{N}$, and $X_1,\ldots,X_n$, $Y$ be normed linear spaces.
Let $A_k\in\mathcal{B}(X_k,Y)$, $1\leqslant k\leqslant n$.

Suppose $Y$ has $M$-cotype $\rho\in[1,\infty)$ with constant $C$.
Let $q\in[\rho,\infty]$.
Define the operator $B:\ell_q(X_1,\ldots,X_n)\to Y$ by
\begin{equation*}
B(x_1,\ldots,x_n)=\sum_{k=1}^n A_k x_k.
\end{equation*}
Then
\begin{equation*}
\|B\|\geqslant C\|(\|A_1\|,\ldots,\|A_n\|)\|_r,
\end{equation*}
where $r\in[\rho,\infty]$ is defined by
\begin{equation*}
\frac{1}{q}+\frac{1}{r}=\frac{1}{\rho}.
\end{equation*}
\end{lemma}
\begin{proof}
If $A_k=0$, $1\leqslant k\leqslant n$, then the required assertion is obvious.
Assume that $A_k\neq 0$ for some $k$.

Fix any $\delta>0$.
There exist $x_k\in X_k$, $1\leqslant k\leqslant n$, such that $\|x_k\|=1$ and $\|A_k x_k\|\geqslant\|A_k\|/(1+\delta)$, $1\leqslant k\leqslant n$.
Let $a_k$, $1\leqslant k\leqslant n$, be nonnegative numbers.
Since $Y$ has $M$-cotype $\rho$ with constant $C$, there exist $\varepsilon_k\in\{\pm 1\}$, $1\leqslant k\leqslant n$, such that
\begin{equation*}
\|\sum_{k=1}^n\varepsilon_k a_k A_k x_{k}\|\geqslant C \|(a_1\|A_1 x_1\|,\ldots,\|a_n A_n x_{n}\|)\|_\rho.
\end{equation*}
Hence,
\begin{equation}\label{E:ineq_B}
\|\sum_{k=1}^n A_k(\varepsilon_k a_k x_k)\|\geqslant\frac{C}{1+\delta}\|(a_1\|A_1\|,\ldots,a_n\|A_n\|)\|_\rho.
\end{equation}
Set
\begin{equation*}
x=(\varepsilon_1 a_1 x_1,\ldots,\varepsilon_n a_n x_n)\in\ell_q(X_1,\ldots,X_n).
\end{equation*}
Using~\eqref{E:ineq_B} we get
\begin{equation}\label{E:ineq_norm_B_1}
\|B\|\geqslant\frac{\|Bx\|}{\|x\|_q}\geqslant
\frac{C}{1+\delta}\frac{\|(a_1\|A_1\|,\ldots,a_n\|A_n\|)\|_\rho}{\|(a_1,\ldots,a_n)\|_q}.
\end{equation}
It follows that
\begin{equation}\label{E:ineq_norm_B_2}
\|B\|\geqslant\frac{C}{1+\delta}\|(\|A_1,\ldots,\|A_n\|)\|_r.
\end{equation}

Indeed, if $q\in(\rho,\infty)$, then $r\in(\rho,\infty)$.
Substituting $a_k=\|A_k\|^{r/q}$, $1\leqslant k\leqslant n$, into~\eqref{E:ineq_norm_B_1} we get~\eqref{E:ineq_norm_B_2}.

If $q=\rho$, then $r=\infty$.
Let $\|A_j\|=\max_{1\leqslant k\leqslant n}\|A_k\|$.
Substituting $a_k=0$, $k\neq j$, $a_j=1$ into~\eqref{E:ineq_norm_B_1}, we get~\eqref{E:ineq_norm_B_2}.

If $q=\infty$, then $r=\rho$.
Substituting $a_k=1$, $1\leqslant k\leqslant n$, into~\eqref{E:ineq_norm_B_1}, we get~\eqref{E:ineq_norm_B_2}.

Since $\delta>0$ was arbitrary, the assertion of the lemma follows from~\eqref{E:ineq_norm_B_2}.
\end{proof}

\section{Proof of Theorem~\ref{T:theorem1}}\label{S:proof_1}

It is sufficient to prove that $\mathcal{D}_p$ is nonempty (see Lemma~\ref{L:lemma_dense_D}).

For $n\geqslant 1$, define the operator $B_n:X\to\ell_p(Y_1,\ldots,Y_n)$ by
\begin{equation*}
B_n x=(A_1 x,\ldots,A_n x).
\end{equation*}
Define $q$ by $1/p+1/q=1$.
Then $B_n^*:\ell_q(Y_1^*,\ldots,Y_n^*)\to X^*$.
It is easy to check that
\begin{equation*}
B_n^*(y_1^*,\ldots,y_n^*)=\sum_{k=1}^n A_k^* y_k^*.
\end{equation*}
Since $1/p-1/r=1-1/\rho$ and $1/p+1/q=1$, we conclude that $1/q+1/r=1/\rho$.
Since $X^*$ has $M$-cotype $\rho$, there exists a constant $C>0$ such that $X^*$ has $M$-cotype $\rho$ with constant $C$.
From Lemma~\ref{L:main_lemma} it follows that
\begin{equation*}
\|B_n^*\|\geqslant C\|(\|A_1^*\|,\ldots,\|A_n^*\|)\|_r.
\end{equation*}
It is well-known that $\|A^*\|=\|A\|$ for any $A\in\mathcal{B}(V,W)$, where $V,W$ are normed linear spaces.
Hence,
\begin{equation*}
\|B_n\|\geqslant C\|(\|A_1\|,\ldots,\|A_n\|)\|_r.
\end{equation*}
Since $(\|A_1\|,\|A_2\|,\ldots)\notin\ell_r$, we conclude that $\|B_n\|\to\infty$ as $n\to\infty$.
From Lemma~\ref{L:unbounded} it follows that there exists $x_0\in X$ such that the sequence $\|B_n x_0\|$ is unbounded.
Clearly, $x_0\in\mathcal{D}_p$.
This completes the proof.

\section{Proof of Theorem~\ref{T:theorem2}}\label{S:proof_2}

It is sufficient to prove that $\bigcap_{p\in[1,p_0)}\mathcal{D}_p$ is nonempty (see Lemma~\ref{L:lemma_dense_D}).

Fix an increasing sequence $p_n\in[1,p_{0})$, $n\geqslant 1$, such that $p_n\to p_0$ as $n\to\infty$.
Let $0=m_0<m_1<m_2<\ldots$ be an increasing sequence of nonnegative integers.
For $n\geqslant 1$ define the operator $B_n:X\to \ell_{p_n}(Y_{m_{n-1}+1},\ldots,Y_{m_n})$ by
\begin{equation*}
B_n x=(A_{m_{n-1}+1}x,\ldots,A_{m_n}x).
\end{equation*}
For $n\geqslant 1$, define $q_n$ by $1/p_n+1/q_n=1$.
Since $p_n<p_0\leqslant\rho/(\rho-1)$, we conclude that $1/p_n>1-1/\rho$, $1/\rho>1/q_n$, $q_n>\rho$.
Clearly, $B_n^*:\ell_{q_n}(Y_{m_{n-1}+1}^*,\ldots,Y_{m_n}^*)\to X^*$.
It is easy to check that
\begin{equation*}
B_n^*(y_{m_{n-1}+1}^*,\ldots,y_{m_n}^*)=\sum_{k=m_{n-1}+1}^{m_n}A_k^* y_k^*.
\end{equation*}
Since $X^*$ has $M$-cotype $\rho$, there exists a constant $C>0$ such that $X^*$ has $M$-cotype $\rho$ with constant $C$.
From Lemma~\ref{L:main_lemma} it follows that
\begin{equation*}
\|B_n^*\|\geqslant C\|(\|A_{m_{n-1}+1}^*\|,\ldots,\|A_{m_n}^*\|)\|_{r_n},
\end{equation*}
where $r_{n}$ is defined by $1/q_n+1/r_n=1/\rho$.
It follows that
\begin{equation}\label{E:th_2_ineq_norm_B_n}
\|B_n\|\geqslant C\|(\|A_{m_{n-1}+1}\|,\ldots,\|A_{m_n}\|)\|_{r_n}.
\end{equation}
Since $1/p_n+1/q_n=1$, we have $1/p_n-1/r_n=1-1/\rho$.
From $p_n<p_0$ it follows that $r_n<r_0$, $n\geqslant 1$.
Since $(\|A_1\|,\|A_2\|,\ldots)\notin\ell_r$ for any $r\in[\rho,r_0)$, we can choose a sequence $m_n$, $n\geqslant 1$, so that
\begin{equation*}
\|(\|A_{m_{n-1}+1}\|,\ldots,\|A_{m_n}\|)\|_{r_n}\to\infty\quad\text{as}\quad n\to\infty.
\end{equation*}
From~\eqref{E:th_2_ineq_norm_B_n} it follows that $\|B_n\|\to\infty$ as $n\to\infty$.
By Lemma~\ref{L:unbounded}, there exists $x_0\in X$ such that the sequence $\|B_n x_0\|$, $n\geqslant 1$, is unbounded.

We claim that
\begin{equation*}
x_0\in\bigcap_{p\in[1,p_0)}\mathcal{D}_p.
\end{equation*}
Indeed, suppose that $\sum_{k=1}^\infty\|A_k x_0\|^p<\infty$ for some $p\in[1,p_0)$.
There exists a number $M$ such that $\sum_{k=M}^\infty\|A_k x_0\|^p<1$.
Let $N$ be such that $m_{N-1}+1\geqslant M$ and $p_N>p$.
For any $n\geqslant N$ we have
\begin{equation*}
\sum_{k=m_{n-1}+1}^{m_n}\|A_k x_0\|^{p_n}<1.
\end{equation*}
Hence, $\|B_n x_0\|<1$ for $n\geqslant N$, a contradiction.
Thus, $x_0\in\bigcap_{p\in[1,p_0)}\mathcal{D}_p$.
This completes the proof.

Ivan Sergeevych Feshchenko,

Taras Shevchenko national university of Kyiv,

Faculty of Mechanics and Mathematics,

Department of Mathematical Analysis,

Assistant professor.

str. Volodymyrska 64, Kyiv, 01033, Ukraine.

Home address: Obolonsky avenue 11, apt.197, Kyiv, 04205, Ukraine.

Phone:+38 044 4133105, +38 095 5931253.

e-mail: ivanmath007@gmail.com

\end{document}